\documentclass[12pt]{article}

\usepackage{amsmath,amssymb,amsthm}

\usepackage{verbatim}

\usepackage{mathtools}

\newcommand{\R}{\mathbb{R}}

\def\Xint#1{\mathchoice
{\XXint\displaystyle\textstyle{#1}}%
{\XXint\textstyle\scriptstyle{#1}}%
{\XXint\scriptstyle\scriptscriptstyle{#1}}%
{\XXint\scriptscriptstyle\scriptscriptstyle{#1}}%
\!\int}
\def\XXint#1#2#3{{\setbox0=\hbox{$#1{#2#3}{\int}$ }
\vcenter{\hbox{$#2#3$ }}\kern-.6\wd0}}

\def\dashint{\Xint-}

\author{Brian Allen}

\date{Fall 2017}

\title{IMCF and the Stability of the PMT and RPI Under $L^2$ Convergence}

\newtheorem{Thm}{Theorem}[section]
\newtheorem{Cor}[Thm]{Corollary}

\newtheorem{Prop}[Thm]{Proposition}
\newtheorem{Lem}[Thm]{Lemma}

\newtheorem{Def}[Thm]{Definition}

\begin{document}

\maketitle

\begin{abstract}
We study the stability of the Positive Mass Theorem (PMT) and the Riemannian Penrose Inequality (RPI) in the case where a region of an asymptotically flat manifold $M^3$ can be foliated by a smooth solution of Inverse Mean Curvature Flow (IMCF) which is uniformly controlled. We consider a sequence of regions of asymptotically flat manifolds $U_T^i\subset M_i^3$, foliated by a smooth solution to IMCF which is uniformly controlled, and if $\partial U_T^i = \Sigma_0^i \cup \Sigma_T^i$ and $m_H(\Sigma_T^i) \rightarrow 0$ then $U_T^i$ converges to a flat annulus with respect to $L^2$ metric convergence. If instead $m_H(\Sigma_T^i)-m_H(\Sigma_0^i) \rightarrow 0$ and $m_H(\Sigma_T^i) \rightarrow m >0$ then we show that $U_T^i$ converges to a topological annulus portion of the Schwarzschild metric with respect to $L^2$ metric convergence.
\end{abstract}

\section{Introduction}\label{sec:intro}

If we consider a complete, asymptotically flat manifold with nonnegative scalar curvature $M^3$ then the Positive Mass Theorem (PMT) says that $M^3$ has positive ADM mass. The rigidity statement says that if $m_{ADM}(M) = 0$ then $M$ is isometric to Euclidean space. Similarly, the Riemannian Penrose Inequality says that if $\partial M$ consists of an  outermost minimal surface $\Sigma_0$ then
\begin{align}
 m_{ADM}(M) \ge \sqrt{\frac{|\Sigma_0|}{16 \pi}}
\end{align}
where $|\Sigma_0|$ is the area of $\Sigma_0$. In the case of equality, i.e. $m_{ADM}(M) = \sqrt{\frac{|\Sigma_0|}{16 \pi}}$, then $M$ is isometric to the Schwarzschild metric. In this paper we are concerned with the stability of these two rigidity statements in the case where we can foliate a region of $M$ by a smooth solution of Inverse Mean Curvature Flow (IMCF) that is uniformly controlled.

The stability problem for the PMT has been studied by Lee \cite{L}, Lee and Sormani \cite{LS1}, Huang, Lee and Sormani \cite{HLS},  LeFloch and Sormani \cite{LeS},  Finster \cite{F}, Finster and Bray \cite{BF},  Finster and Kath  \cite{FK}, and by Corvino \cite{C}. In the work of Lee \cite{L}, Lee considers a sequence of harmonically flat manifolds with ADM mass converging to $0$ and is able to show uniform convergence of the metric outside a ball of a specified radius. In the work of Finster \cite{F}, Bray and Finster \cite{BF}, and Finster and Kath \cite{FK}, spinors are used to obtain $L^2$ estimates of the curvature tensor outside of a set of measure zero. From these estimates stability results are obtained in the sense that the curvature tensor is small in the $L^2$ norm if the mass is small. 

The work presented here is closely related to the work of Lee and Sormani \cite{LS2, LS1}, as well as LeFloch and Sormani \cite{LeS}, where stability of the PMT, stability of the RPI, and compactness properties for Hawking mass are obtained for rotationally symmetric manifolds under intrinsic flat convergence (in fact Lipschitz convergence in the case of \cite{LS2}). In \cite{LS1}, Lee and Sormani conjecture that the PMT should be stable with respect to intrinsic flat convergence for a general class of sequences of asymptotically flat manifolds (See Conjecture 6.2 of \cite{LS1} for details and discussion). In this paper we make an attempt at the general case by showing stability of the PMT and the RPI, with respect to $L^2$ convergence, when our sequence of manifolds can by foliated by a uniformly controlled IMCF. It is still a problem of interest to extend the stability results of this paper to intrinsic flat convergence in order to directly address the conjecture stated in \cite{LS1}. 

The main tool in this paper is IMCF which we remember is defined for surfaces $\Sigma^n \subset M^{n+1}$ evolving through a one parameter family of embeddings $F: \Sigma \times [0,T] \rightarrow M$, $F$ satisfying inverse mean curvature flow

\begin{equation}
\begin{cases}
\frac{\partial F}{\partial t}(p,t) = \frac{\nu(p,t)}{H(p,t)}  &\text{ for } (p,t) \in \Sigma \times [0,T)
\\ F(p,0) = \Sigma_0  &\text{ for } p \in \Sigma 
\end{cases}
\end{equation}
where $H$ is the mean curvature of $\Sigma_t := F_t(\Sigma)$ and $\nu$ is the outward pointing normal vector. The outward pointing normal vector will be well defined in our case since we have in mind, $M^3$, an asymptotically flat manifold with one end.

In \cite{HI}, Huisken and Ilmanen show how to use weak solutions of IMCF in order to prove the RPI in the case of a connected boundary and they note that their techniques give another proof of the  PMT for asymptotically flat Riemanian manifolds when $n=3$ (see Schoen and Yau \cite{SY}, and Witten \cite{W} for more general proofs of the PMT as well as Bray \cite{B} for a more general proof of the RPI). The rigidity results of both the PMT and the RPI are also proved in \cite{HI} and the present work builds off of these arguments by using IMCF to provide a special coordinate system on each member of the sequence of manifolds $M^3_i$ which is leveraged throughout the paper. For a glimpse of long time existence and asymptotic analysis results for smooth IMCF in various ambient manifolds see \cite{CG1,U,CG2,QD,S,BA}.

If we have $\Sigma^2$ a surface in a Riemannian manifold, $M^3$, we will denote the induced metric, mean curvature, second fundamental form, principal curvatures, Gauss curvature, area, Hawking mass and Neumann isoperimetric constant as $g$, $H$, $A$, $\lambda_i$, $K$, $|\Sigma|$, $m_H(\Sigma)$, $IN_1(\Sigma)$, respectively. We will denote the Ricci curvature, scalar curvature, sectional curvature tangent to $\Sigma$, and ADM mass as $Rc$, $R$, $K_{12}$, $m_{ADM}(M)$, respectively.

Now the class of regions of manifolds to which we will by proving stability of the PMT and RPI is defined.

\begin{Def} \label{IMCFClass} Define the class of manifolds with boundary foliated by IMCF as follows
\begin{align*}
\mathcal{M}_{r_0,H_0,I_0}^{T,H_1,A_1}:=\{& U_T \subset M, R \ge 0|
\exists \Sigma \subset M \text{compact, connected surface such that } 
\\& IN_1(\Sigma) \ge I_0, m_H(\Sigma) \ge 0 \text{,and } |\Sigma|=4\pi r_0^2. 
\\ &\exists \Sigma_t \text{ smooth solution to IMCF, such that }\Sigma_0=\Sigma,
\\& H_0 \le H(x,t) \le H_1, |A|(x,t) \le A_1 \text{ for } t \in [0,T],
\\&\text{and } U_T = \{x\in\Sigma_t: t \in [0,T]\} \}
\end{align*}
where $0 < H_0 < H_1 < \infty$, $0 < I_0,A_1,r_0 < \infty$ and $0 < T < \infty$.
\end{Def}
\textbf{Note:} The upper bound on $|A|$ implies an upper bound on $H$ but we make a distinction between these bounds for notational convenience.

\textbf{Note:} One should imagine that $M$ is asymptotically flat in the definition above but we do not need to impose this condition directly since we will be proving stability of compact regions of manifolds $M^i$ in terms of the Hawking mass of the outermost boundary.

Before we state the stability theorems we define some metrics on $\Sigma \times [0,T]$ that will be used throughout this document.
\begin{align}
\delta &= \frac{r_0^2}{4}e^t dt^2 + r_0^2e^t \sigma
\\ g_s &= \frac{r_0^2}{4}\left (1 - \frac{2}{r_0} m e^{-t/2} \right )^{-1} e^tdt^2 + r_0^2e^t \sigma
\\ \hat{g}^i&=\frac{1}{H(x,t)^2}dt^2 + g^i(x,t)
\end{align}
where $\sigma$ is the round metric on $\Sigma$ and $g^i(x,t)$ is the metric on $\Sigma_t^i$. The first metric is the flat metric, the second is the Schwarzschild metric and the third is the metric on $U_T^i$ with respect to the foliation.

\textbf{Note:} These relationships can be observed if we define $s = r_0 e^{t/2}$ then $ds^2 = \frac{r_0^2}{4}e^t dt^2$, $\delta = ds^2 + s^2 d\sigma$ and $g_s = \frac{1}{1 - \frac{2m}{s}}ds^2 + s^2 d\sigma$.

\begin{Thm}\label{SPMT}
Let $U_{T}^i \subset M_i^3$ be a sequence s.t. $U_{T}^i\subset \mathcal{M}_{r_0,H_0,I_0}^{T,H_1,A_1}$ and $m_H(\Sigma_{T}^i) \rightarrow 0$ as $i \rightarrow \infty$.  If we assume one of the following conditions,
\begin{enumerate}
\item $\exists$  $I > 0$ so that $K_{12}^i \ge 0$  and diam$(\Sigma_0^i) \le D$ $\forall$ $i \ge I$, 
\item $\exists$  $[a,b]\subset [0,T]$ such that $\| Rc^i(\nu,\nu)\|_{W^{1,2}(\Sigma\times [a,b])} \le C$ and diam$(\Sigma_t^i) \le D$ $\forall$ $i$, $t\in [a,b]$,where $W^{1,2}(\Sigma\times [a,b])$ is defined with respect to $\delta$,
\end{enumerate}
then
\begin{align}
\hat{g}^i \rightarrow \delta
\end{align}
in $L^2$ with respect to $\delta$.
\end{Thm}

\begin{Thm}\label{SRPI}
Let $U_{T}^i \subset M_i^3$ be a sequence s.t. $U_{T}^i\subset \mathcal{M}_{r_0,H_0,I_0}^{T,H_1,A_1}$, $m_H(\Sigma_{T}^i)- m_H(\Sigma_0^i) \rightarrow 0$ and $m_H(\Sigma_T^i) \rightarrow m > 0$ as $i \rightarrow \infty$. If we assume that $\exists$ $[a,b] \subset [0,T]$ such that $\| Rc^i(\nu,\nu)\|_{W^{1,2}(\Sigma\times [a,b])} \le C$ and diam$(\Sigma_t^i) \le D$ $\forall$ $i$, $t\in [a,b]$, where $W^{1,2}(\Sigma\times [a,b])$ is defined with respect to $\delta$, then
\begin{align}
\hat{g}^i \rightarrow g_s
\end{align}
in $L^2$ with respect to $\delta$.
\end{Thm} 

Now we would like to understand how the above theorems apply to sequences of asymptotically flat manifolds which are foliated by a long time solution of IMCF. For this we will define the special class of asymptotically flat sequences of manifolds that we will be able to deal with in this paper.

\begin{Def}\label{AF}
We say a complete, Riemannian manifold $(M^3,g)$ is an asymptotically flat manifold if there exists $K \subset M$, compact, so that $M \setminus K$ is diffeomorphic to $\R^3\setminus \overline{B}(0,1)$ and so that the metric satisfies
\begin{align}
|g_{ij} - \delta_{ij}|&\le \frac{C}{|x|}\label{MetricCond}
\\|g_{ij,k}| &\le \frac{C}{|x|^2}\label{DerMetricCond}
\\|g_{ij,kl}| \le C &\hspace{0.5cm} |g_{ij,klm}| \le C\label{Der2MetricCond}
\end{align}
as $|x| \rightarrow \infty$ where the derivatives are taken with respect to $\delta$. If $\partial M \not = \emptyset$ then we require $\partial M$ to be an outermost, minimal surface.

We say a sequence of asymptotically flat manifolds $M_j=(M,g_j)$ is uniformly asymptotically flat if the constants in \eqref{MetricCond}, \eqref{DerMetricCond} and \eqref{Der2MetricCond} can be chosen uniformly for the sequence.
\end{Def}

\textbf{Note:} The condition \eqref{Der2MetricCond} is not typically included in the definition of asymptotic flatness and is only used to gain control on derivatives of the Ricci tensor in order to apply Theorems \ref{SPMT} and \ref{SRPI} to prove Corollaries \ref{PMTCOR} and \ref{RPICOR} below.

As a consequence of Theorems \ref{SPMT} and \ref{SRPI} we have the following results when a long time solution exists on a sequence of uniformly asymptotically flat manifolds.
\begin{Cor}\label{PMTCOR}
Assume for all $M_i$ the smooth solution of IMCF starting at $\Sigma_0$ exists for all time, so that for all $T \in (0, \infty)$, $U_{T,i}\subset \mathcal{M}_{r_0,H_0^T,I_0}^{T,H_1,A_1}$, where $H_0^T \rightarrow 0$ as $T \rightarrow \infty$. In addition, we define $\displaystyle m_H(\Sigma_{\infty}^i)=\lim_{T\rightarrow \infty}m_H(\Sigma_{T}^i)$ and assume that $\displaystyle m_H(\Sigma_{\infty}^i) \rightarrow 0$ as $i \rightarrow \infty$ and that $M_i$ are uniformly asymptotically flat with respect to the IMCF coordinates then
\begin{align}
\hat{g}^i \rightarrow \delta
\end{align}
on $\Sigma\times[0,T]$ in $L^2$ with respect to $\delta$.
\end{Cor}

\begin{Cor}\label{RPICOR}
Assume that for all $M_i$ the smooth solution of IMCF starting at $\Sigma_0$ exists for all time, so that $U_{T,i}\subset \mathcal{M}_{r_0,H_0^T,I_0}^{T,H_1,A_1}$ for all $T \in (0, \infty)$, where $H_0^T \rightarrow 0$ as $T \rightarrow \infty$. In addition, we define $\displaystyle m_H(\Sigma_{\infty}) = \lim_{T\rightarrow \infty} m_H(\Sigma_T)$ and assume that $m_H(\Sigma_{\infty}^i)- m_H(\Sigma_{0}^i) \rightarrow 0$, $m_H(\Sigma_{\infty}^i)\rightarrow m > 0$ as $i \rightarrow \infty$, and $M_i$ are uniformly asymptotically flat with respect to the IMCF coordinates then 
\begin{align}
\hat{g}^i \rightarrow g_s
\end{align}
on $\Sigma\times[0,T]$ in $L^2$ with respect to $\delta$.
\end{Cor}

\textbf{Note:} If $\Sigma_0$ is a minimizing hull, this theorem also applies to the regions between jumps of the weak formulation of Huisken and Ilmanen if we stay away from the jump times and condition 1 or 2 of Theorem \ref{SPMT} or condition 1 of Theorem \ref{SRPI} is satisfied. In order to see this it is important to remember three important lemmas of Huisken and Ilmanen

\begin{itemize}

\item Smooth flows satisfy the weak formulation in the domain they foliate (Lemma 2.3 \cite{HI}).

\item The weak evolution of a smooth, $H > 0$, strictly minimizing hull is smooth for a short time (Lemma 2.4 \cite{HI}).

\item It can be shown that the weak solution remains smooth until the first moment when either $\Sigma_t \not= \Sigma_t'$, $H \searrow 0$ or $|A| \nearrow \infty$ where $\Sigma_t'$ is the outward minimizing hull of $\Sigma_t$ (Remark after Lemma 2.4 \cite{HI}). This follows since if $H(x,t) \ge H_0 >0$ and $|A|(x,t) \le A_1 < \infty$ then we can apply regularity results of Krylov \cite{K2} in order to achieve $C^{2,\alpha}$ estimates which then imply a continuation result. If $\Sigma_t$ is outward minimizing then we know that the smooth solution agrees with the weak solution.
\end{itemize}

In the future it would be desirable to extend the results of this paper to weak solutions of IMCF as well as develop a method for dealing with the jump regions which are not foliated by weak IMCF. We now give an outline of the rest of the paper.

In Section 2 we will use IMCF to get important estimates of the metric $\hat{g}$ on the foliated region $U_T^i\subset M_i$. The crucial estimates come from the calculation of the monotonicity of the hawking mass in Lemma \ref{CrucialEstimate} which lead to Corollary \ref{GoToZero}. New estimates of length of geodesics and Neumann isporimetric constants under IMCF are obtained in Lemmas \ref{lengthEvolution} and \ref{IsoControl} which eventually lead to showing that mean curvature of $\Sigma_t^i$ converges to its average in $L^2$.

In Section 3 we use the estimates of the previous section along with some new estimates of the metric on $\Sigma_t$ in Lemma \ref{metricEst} to show convergence of $\hat{g}$ to a warped product $g_3^i(x,t) = \frac{r_0^2e^t}{4} dt^2 + r_0^2e^t g^i(x,0)$. This is done by showing convergence of $\hat{g}$ to simpler metrics, successively, until we get to $g_3^i$ and combining this chain of estimates by the triangle inequality. 

In Section 4 we complete the proofs of Theorems \ref{SPMT} and \ref{SRPI} by showing convergence of $g_3^i$ to $\delta$. This will be done under a few different assumptions on IMCF as well as the curvature of $M_i$. These results are combined with the rigidity result of Petersen and Wei \cite{PW}, Theorem \ref{rigidity}, in order to improve from $L^2$ curvature convergence results to $L^2$ metric convergence.

\vspace{0.5cm}

\textbf{Acknowledgments:}
\vspace{0.15cm}

I would like to thank Christina Sormani for bringing this problem to my attention and for many useful suggestions and discussions. I would also like to thank her for organizing seminars at the CUNY Graduate Center which were important in shaping this paper.

\section{Estimates for Manifolds Foliated by IMCF}\label{Sect-Est}

We start by obtaining some useful estimates where it will be important to remember the definition of the Hawking mass defined for a hypersurface $\Sigma^2 \subset M^3$,
\begin{align}
m_H(\Sigma) = \sqrt{\frac{|\Sigma|}{(16\pi)^3}} \left (16 \pi - \int_{\Sigma} H^2 d \mu \right )
\end{align}

\begin{Lem} \label{naiveEstimate}
Let $\Sigma^2 \subset M^3$ be a hypersurface and $\Sigma_t$ it's corresponding solution of IMCF. If $m_1 \le m_H(\Sigma_t) \le m_2$ then 
\begin{align}
|\Sigma_t|&=|\Sigma_0| e^t
\\16 \pi \left (1 - \sqrt{\frac{16 \pi}{|\Sigma_0|}}m_2e^{-t/2}  \right ) &\le \int_{\Sigma_t} H^2 d \mu \le 16 \pi \left (1 - \sqrt{\frac{16 \pi}{|\Sigma_0|}}m_1e^{-t/2}  \right ) \label{Eq-NaiveAvg}
\\ \frac{16 \pi}{|\Sigma_0|} \left (1 - \sqrt{\frac{16 \pi}{|\Sigma_0|}}m_2e^{-t/2}  \right )e^{-t} &\le \dashint_{\Sigma_t} H^2 d \mu \le  \frac{16 \pi}{|\Sigma_0|} \left (1 - \sqrt{\frac{16 \pi}{|\Sigma_0|}}m_1e^{-t/2}  \right )e^{-t}\label{AvgHEst}
\end{align}
where $|\Sigma_t|$ is the $n$-dimensional area of $\Sigma$.

Hence if $m_{H}(\Sigma_T^i) \rightarrow 0$ then 
\begin{align}
\bar{H^2}_i(t):=\dashint_{\Sigma_t^i} H_i^2 d \mu \rightarrow  \frac{4}{r_0}e^{-t}\label{unifAvgHEst1}
\end{align}
for every $t\in [0,T]$.

If $m_{H}(\Sigma_T^i)-m_{H}(\Sigma_0^i) \rightarrow 0$ and $m_{H}(\Sigma_T^i) \rightarrow m > 0$ then  
\begin{align}
\bar{H^2}_i(t):=\dashint_{\Sigma_t^i} H_i^2 d \mu\rightarrow \frac{4}{r_0^2} \left(1-\frac{2}{r_0}m e^{-t/2}\right)e^{-t} \label{unifAvgHEst2}
\end{align}
for every $t\in [0,T]$.
\end{Lem}

\begin{proof}
Equation \ref{Eq-NaiveAvg} and \ref{AvgHEst} follow directly from the definition of the Hawking mass and the first estimate is standard for IMCF.
The equations \ref{unifAvgHEst1} and \ref{unifAvgHEst2} follow from \ref{AvgHEst} and the assumption on the Hawking mass along the sequence.
\end{proof}

\begin{Lem}\label{dtintEstimate}
For any solution of IMCF we have the following formula
\begin{align}
\frac{d}{dt} \int_{\Sigma_t} H^2 d\mu =  \frac{(16 \pi)^{3/2}}{|\Sigma_t|^{1/2}} \left (\frac{1}{2} m_H(\Sigma_t) - \frac{d}{dt}m_H(\Sigma_t) \right )
\end{align}
So if we assume that $m_H(\Sigma_T^i) \rightarrow 0$ as $i \rightarrow \infty$ then we have for a.e. $t \in [0,T]$ that 
\begin{align}
\frac{d}{dt} \int_{\Sigma_t^i} H^2 d\mu \rightarrow 0\label{Eq-dtH^2PMT}
\end{align}

If we assume that $m_H(\Sigma_T^i) - m_H(\Sigma_0^i) \rightarrow 0$ and $m_H(\Sigma_T^i)\rightarrow m > 0$ as $ i \rightarrow \infty$ then we have that 
\begin{align}
\frac{d}{dt} \int_{\Sigma_t^i} H^2 d\mu \rightarrow \frac{16 \pi}{r_0}me^{-t/2} \label{Eq-dtH^2RPI}
\end{align}
\end{Lem}

\begin{proof}
By using the formula for the hawking mass we can compute that
\begin{align}
\frac{d}{dt}m_H(\Sigma_t)&= \frac{1}{2}m_H(\Sigma_t) - \sqrt{\frac{|\Sigma_t|}{(16\pi)^3}} \frac{d}{dt}\int_{\Sigma} H^2 d \mu \label{Eq-dtInt}
\end{align}
Rearranging this equation by solving for $\frac{d}{dt}\int_{\Sigma} H^2 d \mu $ we find the first formula in the statement of the lemma.

By Geroch Monotonicity we know that $\frac{d}{dt}m_H(\Sigma_t) \ge 0$ and so if $m_H(\Sigma_t^i)\rightarrow 0$ as $i \rightarrow \infty$ then we must have that $ \frac{d}{dt}m_H(\Sigma_t^i) \rightarrow 0$ for almost every $t \in [0,T]$. Combining with \eqref{Eq-dtInt} shows that $\frac{d}{dt}\int_{\Sigma_t^i}H_i^2 d\mu \rightarrow 0$ for almost every $t \in [0,T]$.

If $m_H(\Sigma_T^i) - m_H(\Sigma_0^i) \rightarrow 0$ as $ i \rightarrow \infty$ then we have that $\int_0^T \frac{d}{dt}m_H(\Sigma_t)dt \rightarrow 0$ and so by Geroch monotonicity we must have that $\frac{d}{dt}m_H(\Sigma_t) \rightarrow 0$ for almost every $t \in [0,T]$. Then by combining with the assumption that $m_H(\Sigma_T^i)\rightarrow m$ as $ i \rightarrow \infty$ we get the desired result in this case.
\end{proof}

\begin{Lem}\label{CrucialEstimate}Let $\Sigma^2\subset M^3$ be a compact, connected surface with corresponding solution to IMCF $\Sigma_t$. Then we find the crucial estimate
\begin{align}
m_H(\Sigma_t)\left (\frac{(16 \pi)^{3/2}}{2|\Sigma_t|^{1/2}}\right ) \ge \frac{d}{dt} \int_{\Sigma_t} H^2 d\mu + \int_{\Sigma_t}2\frac{|\nabla H|^2}{H^2} +\frac{1}{2}(\lambda_1-\lambda_2)^2 + R d\mu \label{Eq-Crux}
\end{align}
which can be rewritten and integrated to find
\begin{align}
m_H(\Sigma_T)-m_H(\Sigma_0)  \ge \int_0^T \frac{|\Sigma_t|^{1/2}}{(16 \pi)^{3/2}} \left (\int_{\Sigma_t}2\frac{|\nabla H|^2}{H^2} +\frac{1}{2}(\lambda_1-\lambda_2)^2 + R d\mu \right)dt
\end{align}
\end{Lem}

\begin{proof}
We will use the following facts in the derivation below where $R$ is the scalar curvature of $M$ and $K$ is the Gauss curvature of $\Sigma_t$.
\begin{align}
\frac{R}{2} &= Rc(\nu,\nu) +K - \frac{1}{2}(H^2 - |A|^2)
\\ |A|^2 &= \frac{1}{2}H^2 + \frac{1}{2} (\lambda_1-\lambda_2)^2
\\ \int_{\Sigma_t} K d\mu_t &= 2 \pi \chi(\Sigma_t)
\end{align}
which follow from the Gauss equations, the definition of $|A|^2$ and the Gauss-Bonnet theorem.

Now we compute the time derivative of $\int_{\Sigma_t}H^2 d \mu$
\begin{align}
\frac{d}{dt} \int_{\Sigma_t} H^2 d \mu_t &= \int_{\Sigma_t} 2 H \frac{\partial H}{\partial t} + H^2 d \mu_t
\\&=\int_{\Sigma_t} -2H \Delta \left ( \frac{1}{H} \right ) - 2 |A|^2 -2  Rc(\nu,\nu) + H^2 d \mu_t \label{Eq-dtFirst}
\\ &=\int_{\Sigma_t}-2\frac{|\nabla H|^2}{H^2} -|A|^2 - R + 2K  d \mu_t
\\&= 4 \pi  \chi(\Sigma_t) +\int_{\Sigma_t}-2\frac{|\nabla H|^2}{H^2} - \frac{1}{2}(\lambda_1-\lambda_2)^2 - R  -\frac{1}{2}H^2 d \mu_t
\\&\le m_H(\Sigma_t) \frac{(16 \pi)^{3/2}}{2|\Sigma_t|^{1/2}}+ \int_{\Sigma_t}-2\frac{|\nabla H|^2}{H^2} - \frac{1}{2}(\lambda_1-\lambda_2)^2 - R d \mu_t \label{Eq-lastCruc}
\end{align}
where we are using that $\chi(\Sigma_t)\le 2$ for compact, connected surfaces.
Rearranging \eqref{Eq-lastCruc} we find that
\begin{align}
m_H(\Sigma_t) \frac{(16 \pi)^{3/2}}{|\Sigma_t|^{1/2}} \ge \frac{d}{dt} \int_{\Sigma_t} H^2 d\mu + \int_{\Sigma_t}2\frac{|\nabla H|^2}{H^2} +\frac{1}{2}(\lambda_1-\lambda_2)^2 + R d\mu
\end{align}
Now by combining with Lemma \ref{dtintEstimate} we find 
\begin{align}
\frac{d}{dt}m_H(\Sigma_t)  \ge \frac{|\Sigma_t|^{1/2}}{(16 \pi)^{3/2}} \int_{\Sigma_t}2\frac{|\nabla H|^2}{H^2} +\frac{1}{2}(\lambda_1-\lambda_2)^2 + R d\mu
\end{align}
and then by integrating both sides from $0$ to $T$ we find the desired estimate.
\end{proof}

By combining Lemma \ref{CrucialEstimate} with Lemma \ref{dtintEstimate} we are able to deduce the crucial estimates below which we will show leads to a stability of positive mass theorem.

\begin{Cor} \label{GoToZero}Let $\Sigma^i\subset M^i$ be a compact, connected surface with corresponding solution to IMCF $\Sigma_t^i$. If $m_H(\Sigma_0)\ge0$ and $m_H(\Sigma^i_T) \rightarrow 0$  then for almost every $t \in [0,T]$ we have that
\begin{align}
&\int_{\Sigma_t^i} \frac{|\nabla H_i|^2}{H_i^2}d \mu \rightarrow 0 \hspace{1 cm} \int_{\Sigma_t^i} (\lambda_1^i-\lambda_2^i)^2d \mu \rightarrow 0\hspace{1 cm} \int_{\Sigma_t^i} R^i d \mu \rightarrow 0
\\ &\int_{\Sigma_t^i} Rc^i(\nu,\nu)d \mu \rightarrow 0 \hspace{.6 cm} \int_{\Sigma_t^i} K_{12}^id \mu \rightarrow 0 \hspace{2.1 cm} \int_{\Sigma_t^i} H_i^2 d\mu\rightarrow 16\pi \label{RicciEstimate}
\\&\int_{\Sigma_t^i} |A|_i^2 d \mu \rightarrow 8 \pi\hspace{1 cm} \int_{\Sigma_t^i} \lambda_1^i\lambda_2^i d \mu \rightarrow 4\pi \hspace{2 cm} \chi(\Sigma_t^i) \rightarrow 2
\end{align}
as $i \rightarrow \infty$ where $K_{12}$ is the ambient sectional curvature tangent to $\Sigma_t$. Since $\chi(\Sigma_t^i)$ is discrete we see by the last convergence that $\Sigma_t^i$ must eventually become topologically a sphere. 

If $\left (m_H(\Sigma^i_T)-m_H(\Sigma^i_0) \right ) \rightarrow 0$ where $m_H(\Sigma_T) \rightarrow m > 0$ then the first three integrals listed above $\rightarrow 0$  and for almost every $t \in [0,T]$ we have that
\begin{align}
 &\int_{\Sigma_t^i} H_i^2 d\mu\rightarrow 16 \pi \left (1 - \sqrt{\frac{16 \pi}{|\Sigma_0|}}me^{-t/2}  \right )
 \hspace{0.5 cm} \int_{\Sigma_t^i} |A|_i^2 d \mu \rightarrow 8 \pi \left (1 - \sqrt{\frac{16 \pi}{|\Sigma_0|}}me^{-t/2}  \right ) \hspace{0.5 cm}
 \\& \int_{\Sigma_t^i} \lambda_1^i\lambda_2^i d \mu \rightarrow 4 \pi \left (1 - \sqrt{\frac{16 \pi}{|\Sigma_0|}}me^{-t/2}  \right )\hspace{0.5 cm}\int_{\Sigma_t^i}Rc^i(\nu,\nu)d\mu \rightarrow -\frac{8\pi}{r_0}m e^{-t/2}
 \\ &\int_{\Sigma_t^i}K_{12}^i d\mu \rightarrow \frac{8\pi}{r_0}m e^{-t/2}\hspace{3 cm}\chi(\Sigma_t^i) \rightarrow 2
\end{align}
Since $\chi(\Sigma_t^i)$ is discrete we see by the last convergence that $\Sigma_t^i$ must eventually become topologically a sphere.
\end{Cor}

\begin{proof}
The first three integrals converge to $0$ by Lemma \ref{CrucialEstimate} \eqref{Eq-Crux} so now we will show how to deduce the last three. Using the calculation in \ref{CrucialEstimate}   we can rewrite \eqref{Eq-dtFirst} as
\begin{align}
\frac{d}{dt} \int_{\Sigma_t^i} H_i^2 d \mu_t &=\int_{\Sigma_t^i} -2\frac{|\nabla H_i|^2}{H_i^2} - (\lambda_1^i-\lambda_2^i)^2 -2  Rc^i(\nu,\nu)  d \mu_t
\end{align}
which implies that the integral of $Rc(\nu,\nu) \rightarrow 0$ for almost every $t \in [0,T]$ since every other integral in that expression $\rightarrow 0$ for almost every $t \in [0,T]$. Then we can write

\begin{align}
\int_{\Sigma_t^i} K_{12}^i d\mu &= \int_{\Sigma_t^i} \frac{1}{2} \left (R^i-2  Rc^i(\nu,\nu) \right) d \mu
\end{align}
which implies that the integral of $K_{12}^i \rightarrow 0$ for almost every $t \in [0,T]$. Then going back to \eqref{Eq-dtFirst} we find

\begin{align}
\frac{d}{dt} \int_{\Sigma_t^i} H_i^2 d \mu_t &=\int_{\Sigma_t^i} -2\frac{|\nabla H_i|^2}{H_i^2} - 2|A|_i^2 + H_i^2 -2  Rc^i(\nu,\nu)  d \mu_t \label{Eq-dtInterm}
\end{align}
which implies that for almost every $t \in [0,T]$ we have that
\begin{align}
\int_{\Sigma_t^i} \frac{1}{2}H_i^2 -|A|_i^2 d \mu_t \rightarrow 0
\end{align}
as $i \rightarrow \infty$, which when combined with Lemma \ref{naiveEstimate}\eqref{Eq-NaiveAvg} implies the desired result. 

Lastly we notice
\begin{align}
&\int_{\Sigma_t^i} \lambda_1^i\lambda_2^i d \mu =\int_{\Sigma_t^i} \frac{1}{2}(H_i^2 - |A|_i^2) \rightarrow 4\pi
\end{align}
and so
\begin{align}
2 \pi \chi(\Sigma_t^i) = \int_{\Sigma_t^i} K^i d\mu = \int_{\Sigma_t^i}\lambda_1^i\lambda_2^i + K_{12}^i  d\mu \rightarrow 4\pi
\end{align}
The convergence results if we assume $\left (m_H(\Sigma^i_T)-m_H(\Sigma^i_0) \right ) \rightarrow 0$ follow similarly using Lemma \ref{naiveEstimate} in order to find specifically what $\int_{\Sigma_t^i} |A|_i^2 d \mu$ or $\int_{\Sigma_t^i} \lambda_1^i\lambda_2^i d \mu$ converge to. To find what $Rc^i(\nu,\nu)$ and $K_{12}^i$ converge to we use Lemma \ref{dtintEstimate} \eqref{Eq-dtH^2RPI} which tells us that for almost every $t \in [0,T]$ we have that $\frac{d}{dt}\int_{\Sigma_t^i}H_i^2 d \mu \rightarrow \frac{16 \pi}{r_0} m e^{-t/2}$ and by combining this with \eqref{Eq-dtInterm} we find
\begin{align}
\int_{\Sigma_t^i} Rc^i(\nu,\nu)  d \mu_t&=\int_{\Sigma_t^i} -\frac{|\nabla H_i|^2}{H_i^2} - |A|_i^2 + \frac{1}{2}H_i^2     d \mu_t -\frac{1}{2}\frac{d}{dt} \int_{\Sigma_t^i} H_i^2 d \mu_t
\end{align}
Then since $\int_{\Sigma_t^i}|A|_i^2 d \mu dt \rightarrow -8 \pi  - \frac{8 \pi}{r_0} m e^{-t/2}$ we find that $\int_{\Sigma_t^i} Rc(\nu,\nu)d\mu d_t \rightarrow - \frac{8 \pi}{r_0} m e^{-t/2}$ and hence $\int_{\Sigma_t^i} K_{12}^id\mu d_t \rightarrow  \frac{8 \pi}{r_0} m e^{-t/2}$ and so
\begin{align}
2 \pi \chi(\Sigma_t^i) = \int_{\Sigma_t^i} K^i d\mu = \int_{\Sigma_t^i}\lambda_1^i\lambda_2^i + K_{12}^i  d\mu \rightarrow 4\pi
\end{align}
\end{proof}

In order for the integral quantities above to be useful to us we need to ensure that no collapsing of regions of $\Sigma_t^i$ can occur as $i\rightarrow \infty$. We will accomplish this by proving lower bounds on the isoperimetric constant which we define below. We will also use the sobolev constant to deduce useful information from the integral of the gradient of the mean curvature.

We start by defining the Neumann $\alpha-$Isoperimetric constant and the Neumann $\alpha-$Sobolev constant of a compact manifold without boundary which can be found in Peter Li's book \cite{Li}.

\begin{Def} The Neumann $\alpha-$Isoperimetric constant and the Neumann $\alpha-$Sobolev constant of a compact manifold without boundary are defined as
\begin{align}
IN_{\alpha}(\Sigma)&= \inf_{\mathclap{\substack{\partial S_1 = \gamma = \partial S_2\\ \Sigma= S_1 \cup \gamma \cup S_2}}} \hspace{0.65cm} \frac{L(\gamma)}{\min{\{|S_1|,|S_2|}\}^{1/\alpha}}
\\SN_{\alpha}(\Sigma)&=\inf_{f \in H_{1,1}(\Sigma)} \frac{\int_{\Sigma}|\nabla f| d \mu}{\left ( \inf_{k \in \R} \int_{\Sigma} |f-k|^{\alpha} \right ) ^{1/\alpha}}
\end{align}
where $L(\gamma)$ represents the length of the curve $\gamma$ which separates $\Sigma$ into two pieces $S_1$ and $S_2$.
\end{Def}

Now one can show that the geometric constant and the analytic constant are essentially equivalent. The proof of the following lemma can be found in Peter Li's Geometric Analysis book \cite{Li}, Theorem 9.6 and Corollary 9.7.

\begin{Thm}(Li \cite{Li}) \label{Iso=Sob} Let $\Sigma$ be a compact Riemannian manifold without boundary then we have that
\begin{align}
 \min \{1, 2^{1-1/\alpha} \}SN_{\alpha}(\Sigma) \le IN_{\alpha}(\Sigma)\le \max \{1, 2^{1-1/\alpha} \}SN_{\alpha}(\Sigma)
\end{align}
Also, if we define $\lambda_1(\Sigma)$ to be the first non-zero Neumann eigenvalue for the Laplacian then we find the following bound due to Cheeger
\begin{align}
\lambda_1(\Sigma) \ge \frac{IN_1(\Sigma)^2}{4}
\end{align}
\end{Thm}

Theorem \ref{Iso=Sob} will be useful to us since we will be able to control the isoperimetric constant of $\Sigma_t^i$ using IMCF evolution equations which will then imply control of the Sobolev constant of $\Sigma_t^i$. We start by calculating the evolution of lengths of curves in $\Sigma_t^i$.

\begin{Lem}\label{lengthEvolution}If $\Sigma_t$ is a solution of IMCF where $0 < H_0 \le H(x,t) \le H_1 < \infty$ and $|A|(x,t) \le A_0 < \infty$, and $\gamma(s) \subset \Sigma$ is a smooth, simple, closed curve then
\begin{align}
L^0(\gamma(s)) e^{-\frac{2A_0}{H_0}t} \le L^t(\gamma(s)) \le L^0(\gamma(s)) e^{\frac{2A_0}{H_0}t}
\end{align}
where $L^t(\gamma(s))$ represents the length of $\gamma$ with respect to the metric of $\Sigma_t$.
\end{Lem}
\begin{proof}
Let $\gamma(s) \subset \Sigma$ be a smooth, simple, closed curve and define $L^t(\gamma(s)) = \int_{\gamma} \sqrt{g_t(\gamma',\gamma')}ds$ where $g_t$ is the metric on $\Sigma$ induced from $\Sigma_t \subset M$. Then we calculate the evolution
\begin{align}
\frac{d}{dt} L^t(\gamma(s)) &= \int_{\gamma} \frac{\partial}{\partial t}\sqrt{g_t(\gamma',\gamma')}ds
\\&=\int_{\gamma} \frac{\frac{\partial g_t}{\partial t}}{\sqrt{g_t(\gamma',\gamma')}}ds
\\&=\int_{\gamma} \frac{2A(\gamma',\gamma')}{H\sqrt{g_t(\gamma',\gamma')}}ds
\\&\ge -\int_{\gamma} \frac{2A_0 g(\gamma',\gamma')}{H_0\sqrt{g_t(\gamma',\gamma')}}ds = -\frac{2A_0}{H_0}L^t(\gamma(s))
\end{align}
where the estimate then follows by integrating and the upper bound follows similarly.
\end{proof}

We will now use Lemma \ref{lengthEvolution} in order to control the isoperimetric constant of $\Sigma_t^i$.

\begin{Lem} \label{IsoControl}If $\Sigma_t$ is a solution of IMCF where $0 < H_0 \le H(x,t) \le H_1 < \infty$ and $|A|(x,t) \le A_0 < \infty$ then
\begin{align}
IN_{\alpha}(\Sigma_0) e^{\left (-\frac{2A_0}{H_0}-\frac{1}{\alpha}\right )t} \le IN_{\alpha}(\Sigma_t) \le IN_{\alpha}(\Sigma_0) e^{\left (\frac{2A_0}{H_0}-\frac{1}{\alpha}\right )t}
\end{align}
\end{Lem}
\begin{proof}
Let $\gamma(s) \subset \Sigma$ be a smooth, simple, closed curve and define $L^t(\gamma(s)) = \int_{\gamma} \sqrt{g_t(\gamma',\gamma')}ds$ where $g_t$ is the metric on $\Sigma$ induced from $\Sigma_t \subset M$. Then consider $S \subset \Sigma$ s.t. $\gamma = \partial S$ of which there are two choices and the calculation below will not depend on which choice one makes. We define $S_t:=F_t(S)$ and by the fact that $\frac{\partial}{\partial t} d \mu_t = d\mu_t$ we find that $|S_t| = |S_0|e^t$ as we expect for $|\Sigma_t|$. So we can compute
\begin{align}
\frac{d}{dt} \frac{L^t(\gamma(s))}{|S_t|^{1/\alpha}} & = \frac{\frac{d}{dt}L^t(\gamma(s))}{|S_t|^{1/\alpha}} - \frac{1}{\alpha}\frac{L^t(\gamma(s))}{|S_t|^{1/\alpha}} \ge -\left ( \frac{2A_0}{H_0} + \frac{1}{\alpha} \right) \frac{L^t(\gamma(s))}{|S_t|^{1/\alpha}}
\end{align}
where the estimate 
\begin{align}
\frac{L^0(\gamma(s))}{|S_0|^{1/\alpha}} e^{-\left (\frac{2A_0}{H_0}+\frac{1}{\alpha}\right )t} \le \frac{L^t(\gamma(s))}{|S_t|^{1/\alpha}} \le \frac{L^0(\gamma(s))}{|S_0|^{1/\alpha}} e^{\left (\frac{2A_0}{H_0}-\frac{1}{\alpha}\right )t}
\end{align}
follows by integrating and the upper bound follows similarly. Since this is true for all $\gamma \subset \Sigma$ and all $S^1,S^2\subset \Sigma$ s.t. $\partial S^1 = \gamma = \partial S^2$ and so by taking the $\min{\{|S_t^1|,|S_t^2|\}}$ and then taking the inf over all smooth $\gamma \subset \Sigma$ we find the desired result.
\end{proof}

We will now exploit the newly found control on the isoperimetric constant and hence the sobolev constant to extract useful information from the fact that$\int_{\Sigma_t^i} \frac{|\nabla H|^2}{H^2}d \mu \rightarrow 0$.

\begin{Prop}\label{avgH}If $\Sigma_t^i$ is a sequence of IMCF solutions where $\int_{\Sigma_t^i} \frac{|\nabla H|^2}{H^2}d \mu \rightarrow 0$ as $i \rightarrow \infty$, $0 < H_0 \le H(x,t) \le H_1 < \infty$ and $|A|(x,t) \le A_0 < \infty$ then
\begin{align}
\int_{\Sigma_t^i} (H_i - \bar{H}_i)^2 d \mu \rightarrow 0
\end{align}
as $i \rightarrow \infty$ for almost every $t \in [0,T]$ where $\bar{H}_i = \dashint_{\Sigma_t^i}H_id \mu$. 

Let $d\mu_t^i$ be the volume form on $\Sigma$ w.r.t. $g^i(\cdot,t)$ then we can find a parameterization of $\Sigma_t$ so that 
\begin{align}
d\mu_t^i = r_0^2 e^t d\sigma
\end{align}
where $d\sigma$ is the standard volume form on the unit sphere.

Then for almost every $t \in [0,T]$ and almost every $x \in \Sigma$, with respect to $d\sigma$, we have that 
\begin{align}
H_i(x,t) - \bar{H}_i(t) \rightarrow  0
\end{align}
, along a subsequence. 

\end{Prop}
\begin{proof}
By Lemma \ref{IsoControl} we have uniform control on the isoperimetric constant of $\Sigma_t^i$ and so by Theorem \ref{Iso=Sob} we know that the Sobolev constant of $\Sigma_t^i$ is also controlled and we can use the lower bound on $\lambda_1(\Sigma)$ to control the constant in the Poincare Inequality
\begin{align}
\int_{\Sigma}|\nabla f|^2 \ge \lambda_1(\Sigma)\int_{\Sigma}f^2 d \mu
\end{align}
for $f \in H^{1,2}(\Sigma)$ satisfying $\int_{\Sigma}f d\mu = 0$.

Hence we can calculate
\begin{align}
\int_{\Sigma_t^i} \frac{|\nabla H_i|^2}{H^2_i}d &\ge \frac{1}{H_1^2}\int_{\Sigma_t^i} |\nabla H_i|^2d 
\\&\ge \frac{\lambda_1(\Sigma_t^i)}{H_1^2}\int_{\Sigma_t^i}(H_i - \bar{H}_i)^2 d \mu 
\\&\ge \frac{IN_1(\Sigma_0^i) e^{\left (-\frac{2A_0}{H_0}-1\right )T}}{H_1^2}\int_{\Sigma_t^i}(H_i - \bar{H}_i)^2 d \mu
\\&\ge \frac{I_0 e^{\left (-\frac{2A_0}{H_0}-1\right )T}}{H_1^2}\int_{\Sigma_t^i}(H_i - \bar{H}_i)^2 d \mu
\end{align}
which shows the desired result by applying Lemma \ref{GoToZero}.

Since $\Sigma$ is compact with two measures $d\mu^i_0, r_0^2d \sigma$ of the same area we can use Moser's Theorem \cite{M} to find a diffeomorphism $F^i:S_{r_0} \cong\Sigma\rightarrow \Sigma$ such that for each open set $U \subset \Sigma$ we have that $r_0^2d\sigma(U) = d\mu^i_0(F^i(U))$, i.e. area preserving. Then since $\frac{d}{dt}d\mu_t^i=d\mu_t^i$ we have that $d\mu_t^i = e^t d\mu_0^i$ and if we let $F_t^i$ be the solution of IMCF starting at $F^i$ then $r_0^2e^td\sigma(U) = e^td\mu^i_0(F^i_t(U))=d\mu^i_t(F^i_t(U))$. This means the area preserving diffeomorphism $F^i$ at time $t = 0$ induces an area preserving diffeomorphsim for all times $t \in [0,T]$. 

Then this implies that $\int_0^T\int_{\Sigma}(H_i-\bar{H}_i)^2r_0^2e^td\sigma dt \rightarrow 0$ and hence the pointwise convergence for a.e. $t \in [0,T]$ and for a.e. $x \in \Sigma$, with respect to $d\sigma$, on a subsequence is a well known fact relating $L^2$ convergence to pointwise convergence.
\end{proof}

\textbf{Note:} From now on we will be using the area preserving parameterization, $F_t^i$, of the solution of IMCF, $\Sigma_t$, explained in the proof of \ref{avgH}, which is induced by an area preserving diffeomorphism between $(\Sigma, r_0^2 \sigma)$ and $(\Sigma, g^i(x,0))$.

Now we obtain an estimate which gives us weak convergence of $Rc^i(\nu,\nu)$ which will be used in Section \ref{Sect-End}.

\begin{Lem}\label{WeakRicciEstimate}Let $\Sigma^i_0\subset M^3_i$ be a compact, connected surface with corresponding solution to IMCF $\Sigma_t^i$. Then if $\phi \in C_c^1(\Sigma\times (a,b))$  and $0\le a <b\le T$ we can compute the  estimate
\begin{align}
\int_a^b\int_{\Sigma_t^i}& 2\phi Rc^i(\nu,\nu)d\mu d t=  \int_{\Sigma_a^i} \phi H_i^2 d\mu -  \int_{\Sigma_b^i} \phi H_i^2 d\mu 
\\&+ \int_a^b\int_{\Sigma_t^i}2\phi\frac{|\nabla H_i|^2}{H_i^2}-2\frac{\hat{g}^j(\nabla \phi, \nabla H_i)}{H_i} +\phi(H_i^2-2|A|_i^2) d\mu 
\end{align}

If $m_H(\Sigma^i_T) \rightarrow 0$ and $\Sigma_t$ satisfies the hypotheses of Proposition \ref{avgH} then the estimate above implies 
\begin{align}
\int_a^b\int_{\Sigma_t^i} \phi Rc^i(\nu,\nu)d\mu dt \rightarrow 0
\end{align}
If $m_H(\Sigma^i_T)-m_H(\Sigma^i_0)  \rightarrow 0$, $m_H(\Sigma_T) \rightarrow m > 0$ and $\Sigma_t$ satisfies the hypotheses of Proposition \ref{avgH}  then the estimate above implies 
\begin{align}
\int_a^b\int_{\Sigma_t}& 2\phi Rc^i(\nu,\nu)d\mu d t \rightarrow 2\int_a^b\int_{\Sigma_t} \frac{-8\pi}{r_0}me^{-t/2} \phi d\mu dt
\end{align}
\end{Lem}

\begin{proof}
Let $\phi \in C^1_c(\Sigma\times(a,b))$ and compute
\begin{align}
&\frac{d}{dt} \int_{\Sigma_t^i}\phi H_i^2 d \mu_t = \int_{\Sigma_t^i} 2 \phi H_i \frac{\partial H_i}{\partial t} + \phi H_i^2 + \frac{\partial \phi}{\partial t}H_i^2d \mu
\\&=\int_{\Sigma_t^i} -2\phi H_i \Delta \left ( \frac{1}{H_i} \right ) - 2 \phi |A|_i^2 -2 \phi Rc^i(\nu,\nu) + \phi H^2_i + \frac{\partial \phi}{\partial t}H_i^2 d \mu
\\ &=\int_{\Sigma_t^i}-2\phi \frac{|\nabla H_i|^2}{H^2} -2\frac{\hat{g}^i(\nabla \phi, \nabla H_i)}{H_i} - 2 \phi |A|_i^2 -2 \phi Rc^i(\nu,\nu) + \phi H_i^2 + \frac{\partial \phi}{\partial t}H_i^2 d \mu \label{Eq-lastWeakRicci}
\end{align}
Now by integrating from $a$ to $b$, $0\le a < b \le T$, and rearranging \eqref{Eq-lastWeakRicci} we find that
\begin{align}
\int_a^b\int_{\Sigma_t}& 2\phi Rc^i(\nu,\nu)d\mu d t=  \int_{\Sigma_a} \phi H_i^2 d\mu -  \int_{\Sigma_b} \phi H_i^2 d\mu 
\\&+ \int_a^b\int_{\Sigma_t}2\phi\frac{|\nabla H_i|^2}{H_i^2}-2\frac{\hat{g}^i(\nabla \phi, \nabla H_i)}{H_i} +\phi(H_i^2-2|A|_i^2) + \frac{\partial \phi}{\partial t}H_i^2d\mu 
\end{align}
Notice if $m_H(\Sigma_t^i) \rightarrow 0$ then Proposition \ref{AvgHEst} combined with the assumptions on $\phi$ implies
\begin{align}
\int_a^b\int_{\Sigma_t} \frac{\partial \phi}{\partial t}H_i^2d\mu dt \rightarrow 16\pi\int_a^b \int_{\Sigma}\frac{\partial \phi}{\partial t}d\sigma dt=16\pi \int_{\Sigma}\phi(x,b) - \phi(x,a) d \sigma = 0
\end{align} 

So by using the results of Proposition \ref{AvgHEst} and Corollary \ref{GoToZero} we find that
\begin{align}
\int_a^b\int_{\Sigma_t}& 2\phi Rc^i(\nu,\nu)d\mu d t \rightarrow 0
\end{align}

Notice if $m_H(\Sigma_T^i)-m_H(\Sigma_0^i) \rightarrow 0$ then 
\begin{align}
\int_a^b\int_{\Sigma_t}& \frac{\partial \phi}{\partial t}H_i^2d\mu dt \rightarrow \int_a^b\int_{\Sigma} \frac{\partial \phi}{\partial t}\left (16 \pi - \frac{32\pi }{r_0}me^{-t/2}\right )d\sigma dt 
\\&=\int_a^b\int_{\Sigma} \frac{\partial }{\partial t}\left( \phi\left (16 \pi - \frac{32 \pi }{r_0}me^{-t/2}\right )\right ) -\frac{16 }{r_0}me^{-t/2}\phi d\sigma dt 
\\&= \int_{\Sigma}\phi(x,b)\left (16 \pi - \frac{32\pi }{r_0}me^{-b/2}\right ) - \phi(x,a)\left (16 \pi - \frac{32\pi }{r_0}me^{-a/2}\right )
\\&- \int_a^b\int_{\Sigma}\frac{16\pi }{r_0}me^{-t/2}\phi d\sigma dt 
\end{align}
 by assumption and the convergence of Proposition \ref{AvgHEst}.
  
 So by using the results of Proposition \ref{AvgHEst} and Corollary \ref{GoToZero} we find that  
\begin{align}
\int_a^b\int_{\Sigma_t}& 2\phi Rc^i(\nu,\nu)d\mu d t \rightarrow 2\int_a^b\int_{\Sigma} \frac{-8\pi}{r_0}me^{-t/2} \phi d\sigma dt
\end{align}

\end{proof}

We end this section with an estimate for the metric of $\Sigma_t^i$ in terms of the bounds on the mean curvature and the second fundamental form.

\begin{Lem} \label{metricEst}Assume that $\Sigma_t^i$ is a solution to IMCF and let $\lambda_1^i(x,t)\le \lambda_2^i(x,t)$ be the eigenvalues of $A^i(x,t)$ then we find
\begin{align}
 e^{\int_0^t\frac{2\lambda^i_1(x,s)}{H^i(x,s)}ds} g^i(x,0) \le g^i(x,t) &\le  e^{\int_0^t\frac{2\lambda^i_2(x,s)}{H^i(x,s)}ds} g^i(x,0) 
\end{align}
\end{Lem}
\begin{proof}
We start with the time derivative of the metric 
\begin{align}
\frac{\partial g_{lm}}{\partial t} &= \frac{2A^i_{lm}(x)}{H_i(x)}\le \frac{2\lambda_2^i(x)}{H_0^i(x)}g_{lm} \le \frac{2A_0}{H_0}g_{lm}
\\ \frac{\partial g_{lm}}{\partial t} &= \frac{2A^i_{lm}(x)}{H_i(x)}\ge \frac{2\lambda_1^i(x)}{H_0^i(x)}g_{lm} \ge \frac{-2A_0}{H_0}g_{lm}
\end{align}
where we are fixing the coordinates on $\Sigma_t$ from the time zero hypersurface $\Sigma_0$. By integrating this differential inequality we get the first set of desired estimates.
\end{proof}

\section{Convergence To A Warped Product} \label{Sect-Conv}

In this section we define the following metrics on $\Sigma \times [0,T]$
\begin{align}
\hat{g}^i(x,t) &= \frac{1}{H^i(x,t)^2} dt^2 + g^i(x,t)
\\g_1^i(x,t)&= \frac{1}{\bar{H}^i(t)^2}dt^2 + g^i(x,t)
\\ g_2^i(x,t) &= \frac{1}{\bar{H}^i(t)^2}dt^2 + e^tg^i(x,0)
\\ g_3^i(x,t) = \frac{r_0^2e^t}{4} dt^2 + e^t g^i(x,0)& \hspace{0.25cm}\text{ or } \hspace{0.25cm}
\\g_3^i(x,t) = \frac{r_0^2}{4}&\left (1 - \frac{2}{r_0} m e^{-t/2} \right )^{-1} dt^2 + e^t g^i(x,0)
\\ \delta(x,t) = \frac{r_0^2e^t}{4} dt^2 + r_0^2e^t \sigma(x) &\hspace{0.25cm}\text{ or }\hspace{0.25cm} 
\\g_s(x,t) = \frac{r_0^2}{4}&\left (1 - \frac{2}{r_0} m e^{-t/2} \right )^{-1} dt^2 + r_0^2e^t \sigma(x)
\end{align}
and successively show the pairwise convergence of the metrics in $L^2$ from $\hat{g}^i(x,t)$ to $g_3^i(x,t)$. By combining all the pairwise convergence results using the triangle inequality we will find that $\hat{g}^i -g_3^i \rightarrow 0$ in $L^2$. In the next section we will complete the desired results by showing the convergence to $\delta$ or $g_s$. 

We start by showing that $\hat{g}^i$ converges to $g_1^i$ by using Proposition \ref{avgH}.

\begin{Thm}\label{gtog1} Let $U_{T}^i \subset M_i^3$ be a sequence such that $U_{T}^i\subset \mathcal{M}_{r_0,H_0,I_0}^{T,H_1,A_1}$ and $m_H(\Sigma_{T}^i) \rightarrow 0$ as $i \rightarrow \infty$ or $m_H(\Sigma_{T}^i)- m_H(\Sigma_{0}^i) \rightarrow 0$ and $m_H(\Sigma_T^i)\rightarrow m > 0$. If we define the metrics 
\begin{align}
\hat{g}^i(x,t) &= \frac{1}{H_i(x,t)^2} dt^2 + g^i(x,t)
\\g^i_1(x,t)&= \frac{1}{\overline{H}_i(t)^2}dt^2 + g^i(x,t)
\end{align}
 on $U_T^i$ then we have that
\begin{align}
\int_{U_T^i}|\hat{g}^i -g^i_1|^2 dV \rightarrow 0 
\end{align}
where $dV$ is the volume form on $U_T^i$. 
\end{Thm}

\begin{proof}
We compute
\begin{align}
\int_{U_T^i}|\hat{g}^i -g^i_1|^2 dV &= \int_0^T \int_{\Sigma_t^i} \frac{|\hat{g}^i -g^i_1|^2}{H} d\mu dt 
\\&=\int_0^T \int_{\Sigma_t^i} \frac{1}{H_i} \left | \frac{1}{H_i^2} - \frac{1}{\bar{H}_i^2}\right |^2d\mu dt 
\\&= \int_0^T \int_{\Sigma_t^i} \frac{|\bar{H}_i^2-H_i^2|^2}{H_i^3\bar{H}_i^2} d\mu dt 
\\&\le \frac{1}{H_0^5}\int_0^T \int_{\Sigma_t^i} |\bar{H}_i^2-H_i^2|^2 d\mu dt\rightarrow 0\label{Eq-2Lastgtog1}
\end{align}
where the convergence in \eqref{Eq-2Lastgtog1} follows from the pointwise convergence for almost every $t \in [0,T]$ and almost every $x \in \Sigma_t$, for a subsequence, from Proposition \ref{avgH} as well as the fact that $H_i\le H_1$ and Lebesgue's dominated convergence theorem.

We can get rid of the need for a subsequence by assuming to the contrary that for $\epsilon > 0$ there exists a subsequence so that $\int_{U_T^k}|\hat{g}^k -g^k_1|^2 dV \ge \epsilon$, but this subsequence satisfies the hypotheses of Theorem \ref{gtog1} and hence by what we have just shown we know a subsequence must converge which is a contradiction.
\end{proof}

Now we show the convergence of $g_1^i$ to $g_2^i$.

\begin{Thm}\label{g1tog2} Let $U_{T}^i \subset M_i^3$ be a sequence s.t. $U_{T}^i\subset \mathcal{M}_{r_0,H_0,I_0}^{T,H_1,A_1}$ and $m_H(\Sigma_{T}^i) \rightarrow 0$ as $i \rightarrow \infty$ or $m_H(\Sigma_{T}^i)- m_H(\Sigma_{0}^i) \rightarrow 0$ and $m_H(\Sigma_t^i)\rightarrow m > 0$. If we define the metrics 
\begin{align}
g^i_1(x,t)&= \frac{1}{\overline{H}_i(t)^2}dt^2 + g^i(x,t)
\\g^i_2(x,t)&= \frac{1}{\overline{H}_i(t)^2}dt^2 + e^tg^i(x,0)
\end{align}
 on $U_T^i$ then we have that
\begin{align}
\int_{U_T^i}|g^i_1 -g^i_2|_{g_3^i}^2 dV \rightarrow 0 
\end{align}
where $dV$ is the volume form on $U_T^i$ and the norm is being calculated with respect to the metric $g^i_3(x,t)= \frac{r_0^2}{4}e^tdt^2 + e^tg^i(x,0)$.

Similarly, if we define 
\begin{align}
g^i_{2'}(x,t)= \frac{1}{\overline{H}_i(t)^2}dt^2 + e^{t-T}g^i(x,T)
\end{align}
 on $U_T^i$ then we have that
\begin{align}
\int_{U_T^i}|g^i_1 -g^i_{2'}|_{g_3^i}^2 dV \rightarrow 0 
\end{align}
where $dV$ is the volume form on $U_T^i$.
\end{Thm}

\begin{proof}
We compute
\begin{align}
&\int_{U_T^i}|g^i_1 -g^i_2|^2 dV = \int_0^T \int_{\Sigma_t^i} \frac{|g^i_1 -g^i_2|^2}{H_i} d\mu dt 
\\&=\int_0^T \int_{\Sigma_t^i}e^{-2t}  \frac{|g^i(x,t) -e^tg^i(x,0)|_{g^i(x,0)}^2}{H_i}d\mu dt 
\\&\le \int_0^T \int_{\Sigma_t^i} e^{-2t}\frac{|g^i(x,0)|_{g^i(x,0)}^2}{H_i}\max\{|e^{\int_0^t\frac{2\lambda^i_1(x,s)}{H^i(x,s)}ds} -e^t|^2,|e^{\int_0^t\frac{2\lambda^i_2(x,s)}{H^i(x,s)}ds} -e^t|^2\} d\mu dt 
\\&\le \frac{n^2}{H_0}\int_0^T\int_{\Sigma_t^i} e^{-2t} \max\{|e^{\int_0^t\frac{2\lambda^i_1(x,s)}{H^i(x,s)}ds} -e^t|^2,|e^{\int_0^t\frac{2\lambda^i_2(x,s)}{H^i(x,s)}ds} -e^t|^2\}  d\mu dt \rightarrow 0 \label{Eq-2lastg1tog2}
\end{align} 
where the convergence in \eqref{Eq-2lastg1tog2} follows from Proposition \ref{avgH} since $H_i \rightarrow \bar{H} = \frac{2}{r_0}e^{-t/2}$  and $\lambda_1^i \rightarrow \lambda_2^i$ pointwise almost everywhere with respect to $d \sigma$ along a subsequence. So we have that  $\lambda_p^i(x,t) \rightarrow \frac{1}{r_0}e^{-t/2}$, $p = 1,2$, for almost every $x \in \Sigma_t$ and for almost every $t \in [0,T]$ along a subsequence. This implies that $\frac{2\lambda_p^i(x,t)}{H_i(x,t)} \rightarrow 1$ for almost every $x \in \Sigma_t$ and  for almost every $t \in [0,T]$ along a subsequence. Combining this with the estimate $\frac{2\lambda_p^i}{H_i} \le \frac{A_0}{H_0}$ and Lebesgue's dominated convergence theorem we find the desired convergence above.

We can get rid of the need for a subsequence by assuming to the contrary that for $\epsilon > 0$ there exists a subsequence so that $\int_{U_T^k}|g_1^k -g^k_2|_{g_3^i}^2 dV \ge \epsilon$, but this subsequence satisfies the hypotheses of Theorem \ref{g1tog2} and hence by what we have just shown we know a subsequence must converge which is a contradiction.

We can obtain the convergence result in the case where $m_H(\Sigma_{T}^i)- m_H(\Sigma_{0}^i) \rightarrow 0$ and $m_H(\Sigma_t^i)\rightarrow m$ in a similar fasion by using the estimates of Proposition \ref{avgH} as well as Lemma \ref{GoToZero}.

 Using a similar argument, as well as the time $T$ estimate from Lemma \ref{metricEst}, we can get the second convergence result for $g_{2'}^i$.
\end{proof}

Notice that in Theorem \ref{gtog1} we were able to leverage the results of Proposition \ref{avgH} in order to gain control of the radial portion of the metric $\hat{g}^i$ as $i \rightarrow \infty$. We will further improve on this radial control in Theorem \ref{g2tog3} by using the knowledge that $\bar{H}_i(t)^2 \rightarrow \frac{4}{r_0^2} e^{-t}$ as $i \rightarrow \infty$ to complete the convergence to the warped product $g_3^i$.

\begin{Thm}\label{g2tog3} Let $U_{T}^i \subset M_i^3$ be a sequence s.t. $U_{T}^i\subset \mathcal{M}_{r_0,H_0,I_0}^{T,H_1,A_1}$ and $m_H(\Sigma_{T}^i) \rightarrow 0$ as $i \rightarrow \infty$. If we define the metrics \begin{align}
g^i_2(x,t)&= \frac{1}{\bar{H}^i(t)^2}dt^2 + e^tg^i(x,0)
\\g^i_3(x,t)&= \frac{r_0^2}{4}e^tdt^2 + e^tg^i(x,0)
\end{align}
 on $U_T^i$ then we have that
\begin{align}
\int_{U_T^i}|g^i_2 -g^i_3|^2 dV \rightarrow 0 
\end{align}
where $dV$ is the volume form on $U_T^i$.

Instead, if $m_H(\Sigma_{T}^i)- m_H(\Sigma_{0}^i) \rightarrow 0$ and $m_H(\Sigma_t^i)\rightarrow m > 0$ and we define 
\begin{align}
g^i_3(x,t)= \frac{r_0^2}{4}\left (1 - \frac{2}{r_0} m e^{-t/2} \right )^{-1} e^t dt^2 + e^tg^i(x,0)
\end{align} 
on $U_T^i$ then we have that
\begin{align}
\int_{U_T^i}|g^i_2 -g^i_3|^2 dV \rightarrow 0 
\end{align}
where $dV$ is the volume form on $U_T^i$.
\end{Thm}
\begin{proof}
We calculate
\begin{align}
\int_{U_T^i}|\hat{g}_2^i -g^i_3|^2 dV &= \int_0^T \int_{\Sigma_t^i} \frac{|\hat{g}_2^i -g^i_3|^2}{H} d\mu dt 
\\&=\int_0^T \int_{\Sigma_t^i} \frac{1}{H} \left |  \frac{1}{\bar{H}^2}- \frac{r_0^2}{4}e^{t}\right |d\mu dt 
\\&= \int_0^T \int_{\Sigma_t^i} \frac{r_0^2}{4}e^t\frac{|\frac{4}{r_0^2}e^{-t}-\bar{H}^2|}{H\bar{H}^2} d\mu dt 
\\&\le \frac{r_0^2|\Sigma_0|}{H_0^34}\int_0^T e^{2t} |\frac{4}{r_0^2}e^{-t}-\bar{H}^2| dt\rightarrow 0 \label{Eq-2last1}
\end{align}
where the convergence in \eqref{Eq-2last1} follows from Lemma \ref{naiveEstimate}, \eqref{unifAvgHEst1}.
 
Since this argument is solely concerned with the $dt^2$ part of the metric the argument does not change at all for the convergence of the metrics $g_{2'}^i$ and $g_{3'}^i$. Also, in the case where $m_H(\Sigma_t) \rightarrow m$ the proof is very similar where we use that $\bar{H}^2 \rightarrow \frac{4}{r_0^2}\left (1 - \frac{2}{r_0} m e^{-t/2} \right )$ from Lemma \ref{naiveEstimate}.
\end{proof}

\section{Convergence to Flat/Schwarschild Metric}\label{Sect-End}

In this section we will complete the proofs of Theorems \ref{SPMT} and \ref{SRPI} under a few different assumptions. One should note that the results of the last section are enough to prove Theorems \ref{SPMT} and \ref{SRPI} in the rotationally symmetric case due to the fact that in that case we know that $(\Sigma, g^i(x,t))$ must be a round sphere by assumption. Of course, the stronger Sormani-Wenger Intrinsic Flat (SWIF) convergence has already been shown in the rotationally symmetric case by Lee and Sormani \cite{LS1}, and LeFloch and Sormani \cite{LeS}. It is also interesting that the extra assumptions of Theorems \ref{SPMT} and \ref{SRPI} are not needed for the results of the last section giving $L^2$ convergence to the warped product $g_3^i$ without the $W^{1,2}$ bound on the Ricci curvature.

In the more general case addressed by Theorems \ref{SPMT} and \ref{SRPI} we need to show that $(\Sigma, g^i(x,t))$ converges to a round sphere. In this section we will be able to show that the Gauss curvature of $\Sigma_t^i$ converges to that of a round sphere and so in order to complete the proofs of Theorems \ref{SPMT} and \ref{SRPI} we will need the following rigidity result of Petersen and Wei (\cite{PW}, Corollary 1.5) which allows us to go from, Gauss curvature of $\Sigma_t^i$ converging to a constant, to, $g^i(x,t)$ converging to $r_0^2e^t\sigma(x)$ in $C^{\alpha}$.

\begin{Cor}(Petersen and Wei \cite{PW})\label{rigidity}
Given any integer $n \ge 2$, and numbers $p > n/2$, $\lambda \in \R$, $v >0$, $D < \infty$, one can find $\epsilon = \epsilon(n,p,\lambda, D) > 0$ such that a closed Riemannian $n-$manifold $(\Sigma,g)$ with
\begin{align}
&\text{vol}(\Sigma)\ge v
\\&\text{diam}(\Sigma) \le D
\\& \frac{1}{|\Sigma|} \int_{\Sigma} \|R - \lambda g \circ g\|^p d \mu \le \epsilon(n,p,\lambda,D)\label{Eq-l2curv}
\end{align}
is $C^{\alpha}$, $\alpha < 2 - \frac{n}{p}$ close to a constant curvature metric on $\Sigma$.
\end{Cor}
In our case $n=2$, $p = 2$, $\alpha < 1$ and the Riemann curvature tensor is $R = K g \circ g$, where $g \circ g$ represents the Kulkarni-Nomizu product, and so $\|R - \lambda g \circ g\|^2 = \|g \circ g\|^2 |K - \lambda|^2 = 2^4|K - \lambda|^2$. This shows that we need to verify that the Gauss curvature of $\Sigma_t$ is becoming constant in order to satisfy \eqref{Eq-l2curv} which is exactly what we will be able to show in Theorem \ref{WarpConv} and Corollaries \ref{End1}, \ref{End2}, \ref{LTPMT}, \ref{LTRPI}, \ref{PKPMT}. Then by combining these results with the rigidity result of Petersen and Wei, Theorem \ref{rigidity}, we are able to complete the proofs of Theorems \ref{SPMT} and \ref{SRPI}. 

We start with a theorem which says that if we knew that the warped products $g_3^i$ also had positive scalar curvature then they would have to converge to $\delta$ as $i \rightarrow \infty$ along a subsequence. 

\begin{Thm}\label{WarpConv}
Let $\tilde{g}^i(x,t)= \frac{r_0^2}{4}e^tdt^2 + e^t g^i(x)$ be a sequence of Riemannian metrics defined on $M=[0,T]\times \Sigma$ where $\Sigma$ is topologically a sphere. If  $\tilde{R}^i$ denotes the scalar curvature with respect to $\tilde{g}^i(x)$ and we assume
\begin{align}
\tilde{R}^i \ge 0
\\ \text{diam}(\Sigma, g^i) \le D
\\m_H(\Sigma_t^i)\rightarrow 0
\end{align}
  then 
\begin{align}
\tilde{g}^i\rightarrow \delta
\end{align}
in $C^{\alpha}$, $\alpha < 1$.
\end{Thm}
\begin{proof}
By the assumption that $\tilde{R}^i \ge 0$ and $m_H(\Sigma_t^i)\rightarrow 0$ we can use Lemma \ref{CrucialEstimate} and Corollary \ref{GoToZero} to conclude that $\int_{\Sigma_t} \tilde{R}^i d \mu \rightarrow 0$ and hence $\tilde{R}^i \rightarrow 0$ pointwise a.e. along a subsequence. Now lets rewrite the metric above by performing the change of coordinates $s = r_0e^{t/2}$ which means $\tilde{g}^i(x,t)= ds^2 + \frac{s^2}{r_0^2} g^i(x,0)$. Now we can use the warped product formula from the work of Dobarro and Lami Dozo (\cite{DD}, Theorem 2.1) to express $\tilde{R}^i$ in terms of the scalar curvature of $\Sigma_t$, which is twice the Gauss curvature in this case $2K^i$, and the warping function as follows
\begin{align}
\tilde{R}^i = -\frac{2}{s^2} + \frac{2K^ir_0^2}{s^2}
\end{align}
which by the fact that $\tilde{R}^j \rightarrow 0$ pointwise a.e. along a subsequence we find that $K^j \rightarrow \frac{1}{r_0^2}$ pointwise a.e. along a subsequence. Now we can apply Corollary 1.5 of \cite{PW} which implies that $(\Sigma,g^i)$ is $C^{\alpha}$, $\alpha < 1$, close to a round sphere and hence $\tilde{g}^i$ is $C^{\alpha}$ close to $\delta$. 

Then we can get rid of the need for a subsequence by assuming to the contrary that for $\epsilon > 0$ there exists a subsequence so that $|\tilde{g}^k -\delta|_{C^{\alpha}} \ge \epsilon$ but this subsequence satisfies the hypotheses of Theorem \ref{WarpConv} and hence by what we have just shown we know a subsequence must converge which is a contradiction.
\end{proof}

The issue with using the theorem above is that we don't know that the warped product $g_3^i$ has positive scalar curvature just because $\hat{g}^i$ has positive scalar curvature. This turns out not to be the right approach here but could prove to be useful in a case where one was assured that the warped product $g_3^i$ inherited the positive scalar curvature from $\hat{g}^i$.

Now we prove Theorems \ref{SPMT} under the assumption that $K_{12}^i\ge 0$, the sectional curvature of $M_i$ tangent to $\Sigma_0^i$, for all $i$ which mimics the rotationally symmetric case where the spheres have positive $K_{12}$.

\begin{Cor} \label{PKPMT}
Let $U_{T,i} \subset M_i^3$ be a sequence s.t. $U_{T,i}\subset \mathcal{M}_{r_0,H_0,I_0}^{T,H_1,A_1}$ and $m_H(\Sigma_{T}^i) \rightarrow 0$ as $i \rightarrow \infty$. If in addition we assume that $K^i_{12}(x,0) \ge 0$, the sectional curvature of $M_i^3$ tangent to $\Sigma_0$, then we have that the Gauss curvature of $\Sigma_0$ w.r.t $g^i(x,0)$ will converge to that of a round sphere of radius $r_0$ and
\begin{align}
\hat{g}^i\rightarrow \delta
\end{align}
in $L^2$ with respect to the metric $\delta$.
\end{Cor}
\begin{proof}
By Lemma \ref{GoToZero} we know that $\int_{\Sigma_0^i} K^i_{12} d \mu \rightarrow 0$ and if we know that $K^i_{12} \ge 0$ then we know that $K_{12}^j \rightarrow 0$ pointwise a.e. on a subsequence. Combining this with the fact that $\lambda_1^j\lambda_2^j \rightarrow \frac{2}{r_0^2}$ pointwise a.e. and the fact that $K^j = K_{12}^j + \lambda_1^j\lambda_2^j$ yields the desired result. Now we can apply the result of Petersen and Wei \cite{PW}, Corollary \ref{rigidity} which implies that $(\Sigma,g^i(x,0))$ is $C^{\alpha}$, $\alpha < 1$, close to a round sphere of radius $r_0$. So we can put everything together by noticing
\begin{align}
\int_{U_T} |\hat{g}^i - \delta|_{\delta}^2dV &\le \int_{U_T} |\hat{g}^i - \delta|_{g_3^i}^2 +|(g_3^i)^{lm}(g_3^i)^{pq} -\delta^{lm}\delta^{pq}||\hat{g} - \delta|_{lp} |\hat{g} - \delta|_{mq}dV
\end{align}
where we can show the last term goes to $0$ by using that  $|g_3^i- \delta|_{C^{\alpha}} \rightarrow 0$ as $i \rightarrow \infty$ and noticing that $\int_{U_T}|\hat{g}^i - \delta|_{\delta}^2dV \le C$.

Then we can get rid of the need for a subsequence by assuming to the contrary that for $\epsilon > 0$ there exists a subsequence so that $|\tilde{g}^k -\delta|_{C^{\alpha}} \ge \epsilon$ but this subsequence satisfies the hypotheses of Theorem \ref{WarpConv} and hence by what we have just shown we know a subsequence must converge which is a contradiction.
\end{proof}

Now we will prove Theorems \ref{SPMT} and \ref{SRPI} under the assumption of integral Ricci curvature bounds. For this one should remember that the Sobolev space $W^{1,2}(\Sigma\times[a,b])$ is defined with respect to $\delta$.

\begin{Cor} \label{End1}
Let $U_{T,i} \subset M_i^3$ be a sequence such that $U_{T,i}\subset \mathcal{M}_{r_0,H_0,I_0}^{T,H_1,A_1}$ and $m_H(\Sigma_{T}^i) \rightarrow 0$ as $i \rightarrow \infty$. If $[a,b]\subset [0,T]$ assume that
\begin{align}
\|Rc^i(\nu,\nu)\|_{W^{1,2}(\Sigma\times [a,b])}  \le C
\end{align}
 and $diam(\Sigma_t^i) \le D$ $\forall$ $i$, $t \in [a,b]$ then 
\begin{align}
\hat{g}^i \rightarrow \delta
\end{align}
 in $L^2$ with respect to the metric $\delta$.
\end{Cor}

\begin{proof}
By the assumption that $\|Rc^i(\nu,\nu)\|_{W^{1,2}(\Sigma\times [a,b])} \le C$ we know by Sobolev embedding that a subsequence converges strongly in $L^2(\Sigma\times [a,b])$ to a function $k(x,t) \in L^2(\Sigma\times [a,b])$ , i.e. 
\begin{align}
\int_a^b\int_{\Sigma}|Rc^j(\nu,\nu) - k(x,t)|^2 r_0^2e^td\sigma dt  \rightarrow 0
\end{align}
 By uniqueness of weak limits, combined with Lemma \eqref{WeakRicciEstimate}, we know that 
 \begin{align}
 \int_a^b\int_{\Sigma}|Rc^j(\nu,\nu)|^2 r_0^2e^td\sigma dt \rightarrow 0
 \end{align} 
 and hence a subsequence of $\int_{\Sigma}|Rc^j(\nu,\nu)|^2 r_0^2e^td\sigma \rightarrow 0$ for a.e. $t \in [a,b]$. If we choose some $t_0\in [a,b]$ where the pointwise convergence holds then we have that $\int_{\Sigma}(K_{12}^i)^2 r_0^2d\sigma  \rightarrow 0$ and hence 
\begin{align}
\int_{\Sigma}(K^i-\frac{1}{r_0^2})^2 r_0^2e^{t_0}d\sigma &=\int_{\Sigma}(K_{12}^i+ \lambda_1^i\lambda_2^i - \frac{1}{r_0^2})^2 r_0^2e^{t_0}d\sigma
\\&\le 2\int_{\Sigma}(K_{12}^i)^2+(\lambda_1^i\lambda_2^i - \frac{1}{r_0^2})^2 r_0^2e^{t_0}d\sigma \rightarrow 0 \label{Eq-lastEnd1}
\end{align}

This shows that $\int_{\Sigma} (K^i - \frac{1}{r_0^2})^2 r_0^2e^{t_0}d\sigma \rightarrow 0$ and hence by combining with the diameter bound diam$(\Sigma_0^i)\le D$   we can apply the rigidity result of Petersen and Wei \cite{PW}, Corollary \ref{rigidity}, with $p = 2$ which implies that $|g^i(x,0) - r_0^2\sigma(x)|_{C^{\alpha}} \rightarrow 0$ as $i \rightarrow \infty$ where $\alpha < 1$. This shows that $|g_3^i- \delta|_{C^{\alpha}} \rightarrow 0$ as $i \rightarrow \infty$ where $\alpha < 1$ which also implies $\int_{U_T}|\hat{g} - \delta|_{g_3^i}dV \rightarrow 0$ as $i \rightarrow \infty$. So we can put everything together by noticing
\begin{align}
\int_{U_T} |\hat{g}^i - \delta|_{\delta}^2dV &\le \int_{U_T} |\hat{g}^i - \delta|_{g_3^i}^2 +|(g_3^i)^{lm}(g_3^i)^{pq} -\delta^{lm}\delta^{pq}||\hat{g} - \delta|_{lp} |\hat{g} - \delta|_{mq}dV \label{last EndEnd1}
\end{align}
where we can show the last term of \eqref{last EndEnd1} goes to $0$ by using that  $|g_3^i- \delta|_{C^{\alpha}} \rightarrow 0$ as $i \rightarrow \infty$ and noticing that $\int_{U_T}|\hat{g}^i - \delta|_{\delta}^2dV \le C$.

Then we can get rid of the need for a subsequence by assuming to the contrary that for $\epsilon > 0$ there exists a subsequence so that $\int_{U_T} |\hat{g}^k - \delta|_{\delta}^2dV \ge \epsilon$, but this subsequence satisfies the hypotheses of Theorem \ref{End1} and hence by what we have just shown we know a further subsequence must converge which is a contradiction.
\end{proof}

Now we finish up by proving a similar theorem in the Riemannian Penrose Inequality case.

\begin{Cor}\label{End2}
Let $U_{T,i} \subset M_i^3$ be a sequence s.t. $U_{T,i}\subset \mathcal{M}_{r_0,H_0,I_0}^{T,H_1,A_1}$, $m_H(\Sigma_{T}^i)-m_H(\Sigma_0^i) \rightarrow 0$, and $m_H(\Sigma_T) \rightarrow m > 0$ as $i \rightarrow \infty$. If $[a,b]\subset [0,T]$ assume that
\begin{align}
\|Rc^i(\nu,\nu)\|_{W^{1,2}(\Sigma\times [a,b])}  \le C
\end{align}
 and $diam(\Sigma_t^i) \le D$ $\forall$ $i$, $t\in [a,b]$ then
\begin{align}
\hat{g}^i \rightarrow g_s
\end{align}
in $L^2$ with respect to the metric $g_s$.
\end{Cor}

\begin{proof}
Now one can repeat the proof of Theorem \ref{End1} in order to finish the proof of the results for Theorem \ref{End2}.
\end{proof}

\textbf{Note:} If $Rc(\nu,\nu) = -\frac{2}{r_0^3}e^{-3t/2}m$ and we let $s = r_0e^{t/2}$ then we see that $Rc(\nu,\nu) = -\frac{2}{s^3}m$ which is what we expect for the Schwarschild metric. 

\textbf{Note:} We could have assumed $W^{1,2}$ bounds on $K_{12}$ on $\Sigma\times [a,b]$, instead of on $Rc(\nu,\nu)$, in order to prove the same results as Corollaries \ref{End1} and \ref{End2}.

Next we will prove Corollaries \ref{PMTCOR} and \ref{RPICOR} under the assumption of long time existence by applying Corollaries \ref{End1} and \ref{End2}, respectively.

\begin{Cor}\label{LTPMT}
Let $U_{\infty,i}=\displaystyle\bigcup_{T \in (0,\infty)} U_{T,i} \subset M^3_i$ be a sequence of asymptotically flat manifolds  such that $U_{T,i}\subset \mathcal{M}_{r_0,H_0^T,I_0}^{T,H_1,A_1}$ for all $T \in (0, \infty)$ where $H_0^T \rightarrow 0$ as $T \rightarrow \infty$. Assume that $\displaystyle m_H(\Sigma_{\infty}^i)=\lim_{T\rightarrow \infty}m_H(\Sigma_{T}^i) \rightarrow 0$ as $i \rightarrow \infty$  and that $M_i$ are uniformly asymptotically flat with respect to the IMCF coordinates. Then there exists a $T_* < \infty$ so that for all $T \ge T_*$ we have
\begin{align}
\hat{g}^i \rightarrow \delta 
\end{align}
on $\Sigma\times[0,T]$ in $L^2$ with respect to the metric $\delta$.
\end{Cor}

\begin{proof}
By the assumption that $M_j$ are uniformly asymptotically flat with respect to the IMCF coordinates we immediately find that diam$(\Sigma_T^i) \le C$ and we can use $\eqref{Der2MetricCond}$ to show that there exists a $\epsilon > 0$ so that $\|Rc^i(\nu,\nu)\|_{W^{1,2}(\Sigma\times [T-\epsilon,T])} \le C$  for all $T \ge T_*$. Then we can apply the results of Corollary \ref{End1} for each fixed $T \ge T_*$ to finish the proof.
\end{proof}

\begin{Cor}\label{LTRPI}
Let $U_{\infty,i}=\displaystyle\bigcup_{T \in (0,\infty)} U_{T,i} \subset M^3_i$ be a sequence of asymptotically flat manifolds  such that $U_{T,i}\subset \mathcal{M}_{r_0,H_0^T,I_0}^{T,H_1,A_1}$ for all $T \in (0, \infty)$ where $H_0^T \rightarrow 0$ as $T \rightarrow \infty$. Define $\displaystyle m_H(\Sigma_{\infty}) = \lim_{T\rightarrow \infty} m_H(\Sigma_T)$ and assume that $m_H(\Sigma_{\infty}^i)- m_H(\Sigma_{0}^i) \rightarrow 0$, $m_H(\Sigma_0^i)\rightarrow m > 0$ as $i \rightarrow \infty$, and that $M_i$ are uniformly asymptotically flat with respect to the IMCF coordinates. Then there exists a $T_* < \infty$ so that for all $T \ge T_*$ we have 
\begin{align}
\hat{g}^i \rightarrow g_s
\end{align}
on $\Sigma\times[0,T]$ in $L^2$ with respect to the metric $\delta$.
\end{Cor}

\begin{proof}
Use the exact same argument as in the proof of Corollary \ref{LTPMT}.
\end{proof}

\end{document}